\documentclass[a4paper]{amsart}
\usepackage{graphicx}
\usepackage{amsmath}
\usepackage{amssymb}
\usepackage{oldgerm}
\usepackage[all]{xy}

\newcommand{\Hom}{\operatorname{Hom}\nolimits}
\renewcommand{\Im}{\operatorname{Im}\nolimits}

\newcommand{\Ker}{\operatorname{Ker}\nolimits}

\newcommand{\Ann}{\operatorname{Ann}\nolimits}

\newcommand{\Ext}{\operatorname{Ext}\nolimits}

\newcommand{\m}{\operatorname{\mathfrak{m}}\nolimits}
\newcommand{\az}{\operatorname{\mathfrak{a}}\nolimits}

\newcommand{\V}{\operatorname{V}\nolimits}

\newtheorem{theorem}{Theorem}[section]

\newtheorem{corollary}[theorem]{Corollary}

\newtheorem{lemma}[theorem]{Lemma}

\theoremstyle{definition}

\theoremstyle{definition}

\theoremstyle{remark}

\theoremstyle{definition}
\newtheorem*{remarks}{Remarks}
\theoremstyle{definition}

\begin{document}
\title{On Support Varieties for Modules over Complete Intersections}
\author{Petter Andreas Bergh}
\address{Petter Andreas Bergh \newline Institutt for matematiske fag \\
NTNU \\ N-7491 Trondheim \\ Norway}

\email{bergh@math.ntnu.no}



 \maketitle

\begin{abstract}
Let $(A, \m, k)$ be a complete intersection of codimension $c$, and
$\tilde{k}$ the algebraic closure of $k$. We show that every
homogeneous algebraic subset of $\tilde{k}^c$ is the cohomological
support variety of an $A$-module, and that the projective
variety of a complete indecomposable maximal Cohen-Macaulay
$A$-module is connected.
\end{abstract}

\section{Introduction}

\sloppy Support varieties for modules over complete intersections
were defined by L.\ Avramov in \cite{Avramov1}, and L.\ Avramov and
R.-O.\ Buchweitz showed in \cite{Avramov2} that these varieties to a large
extent behave precisely like the cohomological varieties of modules
over group algebras of finite groups. Further illustrating this are
the two main results in this paper, the first of which says that
every homogeneous variety is realized as the variety of some
module. The second is a version of J.\ Carlson's
result \cite[Theorem 1']{Carlson} on varieties for modules over
group algebras of finite groups. Namely, we prove that if the
variety of a module decomposes as the union of two closed
subvarieties having trivial intersection, then the (completion of
the) minimal maximal Cohen-Macaulay approximation of the module
decomposes accordingly.

Throughout this paper we let $(A, \m, k)$ be a commutative
Noetherian local complete intersection, i.e.\ the completion
$\widehat{A}$ of $A$ with respect to the $\m$-adic topology is the
residue ring of a regular local ring modulo an ideal generated by
a regular sequence. We denote by $c$ the codimension of $A$, that
is, the integer $\mu ( \m) - \dim A$, where $\mu ( \m )$ is the
minimal number of generators for $\m$. All modules are assumed to
be finitely generated.

We now recall the definition of support varieties for modules over
complete intersections; details can be found in \cite[Section
1]{Avramov1} and \cite[Section 2]{Avramov2}. Let $\widehat{A} [
\chi_1, \dots, \chi_c ]$ be the
polynomial ring in the $c$ commuting Eisenbud operators of
cohomological degree $2$. For every $\widehat{A}$-module $X$ there is
a homomorphism $\widehat{A} [ \chi_1, \dots,
\chi_c ] \xrightarrow{\phi_X} \Ext_{\widehat{A}}^* ( X,X )$ of
graded rings under which $\Ext_{\widehat{A}}^* ( X,Y)$ is
a finitely generated graded $\widehat{A} [ \chi_1, \dots, \chi_c
]$-module for any $\widehat{A}$-module $Y$. Using the canonical
isomorphism $k [ \chi_1, \dots, \chi_c ] \simeq \widehat{A} [ \chi_1,
\dots, \chi_c ] \otimes_{\widehat{A}} k$ we obtain a homomorphism $k [
\chi_1, \dots, \chi_c ] \xrightarrow{\phi_X \otimes 1}
\Ext_{\widehat{A}}^* (
X,X ) \otimes_{\widehat{A}} k$ of graded rings under which
$\Ext_{\widehat{A}}^* ( X,Y) \otimes_{\widehat{A}} k$ is finitely
generated over $k [ \chi_1, \dots, \chi_c ]$. We denote the polynomial
ring $k [ \chi_1, \dots, \chi_c ]$ by $H$ and the graded $H$-module
$\Ext_{\widehat{A}}^* ( X,Y) \otimes_{\widehat{A}} k$ by
$E(X,Y)$. Furthermore, we denote the sequence $\chi_1, \dots, \chi_c$
of Eisenbud operators by $\chi$, so that $\widehat{A}[ \chi ]$ and $k[
\chi ]$ are short-hand notations for $\widehat{A} [ \chi_1, \dots,
\chi_c ]$ and $k [ \chi_1, \dots, \chi_c ]$, respectively.

Let $M$ be an
$A$-module and $\widehat{M} = \widehat{A} \otimes_A M$ its $\m$-adic
completion. The \emph{support variety} $\V(M)$ of $M$ is the
algebraic set
$$\V(M) = \{ \alpha = ( \alpha_1, \dots, \alpha_c ) \in
\tilde{k}^c \mid f( \alpha ) =0 \text{ for all } f \in \Ann_H
E(\widehat{M},\widehat{M}) \},$$
where $\tilde{k}$ is the algebraic closure of $k$. This is equal to
the algebraic set
defined by the annihilator in $H$ of $E(\widehat{M},k)$.

For an ideal $\az$ of $H$ we denote by $\V_H( \az )$ the algebraic
set in $\tilde{k}^c$ defined by $\az$, i.e.\
$$\V_H( \az ) = \{ \alpha = ( \alpha_1, \dots, \alpha_c ) \in
\tilde{k}^c \mid f( \alpha ) =0 \text{ for all } f \in \az \}.$$
Note that the variety $\V (M)$ of $M$ is the set $\V_H \left (
\Ann_H E(\widehat{M},\widehat{M}) \right )$, and if $f$ is an
element of $H$ then $\V_H (f)$ is the set of all elements in
$\tilde{k}^c$ on which $f$ vanishes.

\section{Realizing support varieties}

Before proving the main results we need some notation. Let $R$ be
a commutative Noetherian local ring and $X$ an $R$-module with
minimal free resolution
$$\cdots \to P_2 \to P_1 \to P_0 \to X \to 0,$$
and denote by $\Omega_R^n(X)$ the $n$'th syzygy of $X$. For an
$R$-module $Y$, a homogeneous element $\eta \in \Ext_R^*(X,Y)$ can
be represented by a map $f_{\eta} \colon \Omega_R^{|\eta|}(X) \to
Y$, giving the pushout diagram
$$\xymatrix{
0 \ar[r] & \Omega_R^{|\eta|}(X) \ar[r] \ar[d]^{f_{\eta}} &
P_{|\eta|-1} \ar[r] \ar[d] & \Omega_R^{|\eta|-1}(X) \ar[r]
\ar@{=}[d] & 0 \\
0 \ar[r] & Y \ar[r] & K_{\eta} \ar[r] & \Omega_R^{|\eta|-1}(X)
\ar[r] & 0 }$$ with exact rows. Note that the module $K_{\eta}$ is
independent, up to isomorphism, of the map $f_{\eta}$ chosen as a
representative for $\eta$. The construction of this module first
appeared in the paper \cite{Avramov3} by L.\ Avramov, V.\ Gasharov and
I.\ Peeva, where it is used in the proof of Theorem 7.8.

If $\theta \in \Ext_R^*(X,X)$ is another
homogeneous element, then the Yoneda product $\eta \theta \in
\Ext_R^*(X,Y)$ is a homogeneous element of degree $|\eta| +
|\theta|$. The following lemma links $K_{\eta}$ and $K_{\theta}$ to
$K_{\eta \theta}$ via a short exact sequence, and will be a key
ingredient in the proof of the decomposition theorem in the next
section.

\begin{lemma}[\protect{\cite[Lemma 2.3]{Bergh}}]\label{link}
If $\theta \in \Ext_R^* (X,X)$ and $\eta \in \Ext_R^* (X,Y)$ are two
homogeneous elements, then there exists an exact sequence
$$0 \to \Omega_A^{|\eta|}(K_{\theta}) \to K_{\eta
\theta}\oplus F \to K_{\eta} \to 0$$ of $R$-modules, where $F$ is
free.
\end{lemma}

Now suppose $R$ is Gorenstein and $X$ is a maximal Cohen-Macaulay
(or ``MCM" from now on) module. Then there exists a \emph{complete
resolution}
$$\mathbb{P} \colon \cdots \to P_2 \to P_1 \to P_0 \xrightarrow{d}
P_{-1} \to P_{-2} \to \cdots$$ of $X$, i.e.\ a doubly infinite exact
sequence of free modules in which $\Im d$ is isomorphic to $X$. For
an integer $n \in \mathbb{Z}$ the \emph{stable cohomology module}
$\widehat{\Ext}_R^n(X,Y)$ is defined as the $n$'th homology of the
complex $\Hom_R(\mathbb{P},Y)$. If $X$ and $Y$ are
$\widehat{A}$-modules and $X$ is MCM, then
$\widehat{\Ext}_{\widehat{A}}^*(X,Y) = \bigoplus_{i= -
\infty}^{\infty} \widehat{\Ext}_{\widehat{A}}^i(X,Y)$ is a module
over the ring $\widehat{A} [ \chi ]$ of cohomology
operators, and the exact same proof as the one used to prove
\cite[Lemma 4.2]{Erdmann} shows that for any prime ideal
$\mathfrak{q} \neq (\chi)$ of $\widehat{A} [
\chi ]$ the $\widehat{A} [ \chi
]_{\mathfrak{q}}$-modules $\widehat{\Ext}_{\widehat{A}}^*(X,Y)_{\mathfrak{q}}$ and
$\Ext_{\widehat{A}}^*(X,Y)_{\mathfrak{q}}$ are isomorphic.

We are now ready to prove the first result, which shows that when
``cutting down" the variety of an MCM $\widehat{A}$-module by a
homogeneous element, the resulting homogeneous algebraic set is also
the variety of an $\widehat{A}$-module.

\begin{theorem}\label{cuttingvariety}
Let $\eta \in H^+ =(\chi)$ be a homogeneous
element, and let $\overline{\eta} \in \widehat{A} [
\chi ]$ be a homogeneous element such that
$\overline{\eta} \otimes 1$ corresponds to $\eta$ under the
isomorphism $H \simeq \widehat{A} [ \chi ]
\otimes_{\widehat{A}} k$. Furthermore, let $X$ be an
$\widehat{A}$-module, and let $\theta \in
\Ext_{\widehat{A}}^* (X,X)$ be any homogeneous element such that $\theta
\otimes 1 = \phi_X ( \overline{\eta} ) \otimes 1$ in $E(X,X) =
\Ext_{\widehat{A}}^* (X,X) \otimes_{\widehat{A}} k$.
Then there is an inclusion
$$\V(K_{\theta}) \subseteq \V(X) \cap \V_H (\eta),$$
and equality holds whenever $X$ is MCM.
\end{theorem}

\begin{proof}
Consider the exact sequence
$$0 \to X \to K_{\theta} \to \Omega_{\widehat{A}}^{|\eta|-1}(X)
\to 0$$ representing $\theta$. Since varieties are invariant under
syzygies we have $\V(K_{\theta}) \subseteq \V(X)$, and so the first
half of the theorem will follow if we can establish the inclusion
$\V(K_{\theta}) \subseteq \V_H(\eta)$.

The exact sequence induces a long exact sequence
$$\xymatrix@R=0.2pc@C=1pc{
0 \ar[r] & \Hom_{\widehat{A}}( \Omega_{\widehat{A}}^{|\eta|-1}(X),k)
\ar[r] & \Hom_{\widehat{A}}(K_{\theta},k) \ar[r] &
\Hom_{\widehat{A}}(X,k) \\
\ar[r]^<<<<<<{\circ \theta} & \Ext_{\widehat{A}}^{|\eta|}(M,k) \ar[r] &
\Ext_{\widehat{A}}^1(K_{\theta},k) \ar[r] &
\Ext_{\widehat{A}}^1(X,k) \\
\ar[r]^<<<<<<{\circ (- \theta )} & \Ext_{\widehat{A}}^{|\eta|+1}(M,k)
\ar[r] & \Ext_{\widehat{A}}^2(K_{\theta},k) \ar[r] &
\Ext_{\widehat{A}}^2(X,k) \\
& \vdots & \vdots & \vdots \\
\ar[r]^<<<<<<{\circ (-1)^i \theta } &
\Ext_{\widehat{A}}^{|\eta|+i}(M,k)
\ar[r] & \Ext_{\widehat{A}}^{i+1}(K_{\theta},k) \ar[r] &
\Ext_{\widehat{A}}^{i+1}(X,k) \\
& \vdots & \vdots & \vdots }$$
in cohomology, from which we obtain the short exact sequence
$$0 \to \frac{\Ext_{\widehat{A}}^{*+ |\eta|} (X,k)
}{\Ext_{\widehat{A}}^* (X,k) \circ \theta} \xrightarrow{f}
\Ext_{\widehat{A}}^{*+1}(K_{\theta},k) \xrightarrow{g} \Ker ( \circ \theta
)|_{\Ext_{\widehat{A}}^{*+1}(X,k)} \to 0.$$
Now for any $\widehat{A}$-modules $W$ and $Z$ the left and right
scalar actions from $\widehat{A}[ \chi ]$ on
$\Ext_{\widehat{A}}^*(W,Z)$, through the ring homomorphisms $\phi_Z$
and $\phi_W$, respectively, are actually equal (see
\cite[1.1.2]{Avramov2}). Consequently $\Ext_{\widehat{A}}^* (X,k)
\circ \theta$ is an $\widehat{A}[ \chi ]$-submodule of
$\Ext_{\widehat{A}}^{*+|\eta|}(X,k)$, and the above short
exact sequence is a sequence of  $\widehat{A}[ \chi ]$-modules and
maps. Moreover, the end terms are both annihilated by the element
$\overline{\eta}$. To see this, note that since $\theta
\otimes 1 = \phi_X ( \overline{\eta} ) \otimes 1$ in
$\Ext_{\widehat{A}}^* (X,X) \otimes_{\widehat{A}} k$, the element
$\phi_X ( \overline{\eta} ) - \theta \in \Ext_{\widehat{A}}^* (X,X)$
can be written as a finite sum
$$\phi_X ( \overline{\eta} ) - \theta = \sum m_i \theta_i,$$
where $m_i \in \widehat{\m}$ and $\theta_i \in \Ext_{\widehat{A}}^*
(X,X)$. If $G^* = \oplus_{i=0}^{\infty} G^i$ is any graded right
$\Ext_{\widehat{A}}^* (X,X)$-module annihilated by $\theta$, and with
the property that each graded part $G^i$ is finitely generated over
$\widehat{A}$, then
$$\left [ G^i \cdot \phi_X ( \overline{\eta} ) \right ]
\otimes_{\widehat{A}} k = \left [ G^i
\cdot ( \theta + \sum m_i \theta_i ) \right ] \otimes_{\widehat{A}} k =0.$$
This implies that $G^i \cdot \phi_X ( \overline{\eta} )$ vanishes
itself, hence $\overline{\eta}$ annihilates $G^*$. In particular, the
element $\overline{\eta}$ annihilates the end terms in the above short
exact sequence.

Now for any $i \geq 0$, let $w$ be an element of
$\Ext_{\widehat{A}}^i(K_{\theta},k)$, and consider the element
$\overline{\eta} \cdot w \in
\Ext_{\widehat{A}}^{i+|\eta|}(K_{\theta},k)$. Since $g(
\overline{\eta} \cdot w) = \overline{\eta} \cdot g(w) =0$, there must
exist an element $z \in \frac{\Ext_{\widehat{A}}^{*+ |\eta|} (X,k)
}{\Ext_{\widehat{A}}^* (X,k) \circ \theta}$ with the property that
$\overline{\eta} \cdot w = f(z)$, giving $\overline{\eta}^2 \cdot w =
f ( \overline{\eta} \cdot z)=0$. Therefore the element
$\overline{\eta}^2$ annihilates $\Ext_{\widehat{A}}^i(K_{\theta},k)$, and
so the element $\eta^2 \in H$ is contained in $\Ann_H
E(K_{\theta},k)$. This gives the inclusion $\V (K_{\theta})
\subseteq \V_H ( \eta^2 ) = \V_H ( \eta )$, thereby establishing the first half
of the theorem.

\sloppy Next suppose that $X$ is MCM, and let $\mathfrak{p} \neq H^+$
be a prime ideal of $H$ containing $\eta$ and $\Ann_H E(X,k)$. Choose
a prime ideal $\overline{\mathfrak{p}} \neq (\chi)$ of $\widehat{A} [
\chi ]$ corresponding to $\mathfrak{p}$ and containing
$\overline{\eta}$ and the annihilator of $\Ext_{\widehat{A}}^*
(X,k)$, and suppose $\overline{\mathfrak{p}}$ does
\emph{not} contain the annihilator of $\Ext_{\widehat{A}}^*
(K_{\theta},k)$.
The exact sequence from the beginning of the proof
induces a long exact sequence
$$\cdots \to \widehat{\Ext}_{\widehat{A}}^i(K_{\theta},k) \to
\widehat{\Ext}_{\widehat{A}}^i(X,k) \xrightarrow{\circ (-1)^i \theta}
\widehat{\Ext}_{\widehat{A}}^{i+|\eta|}(X,k) \to
\widehat{\Ext}_{\widehat{A}}^{i+1}(K_{\theta},k) \to \cdots$$ in
stable cohomology, which in turn gives the exact sequence
$$0 \to \frac{\widehat{\Ext}_{\widehat{A}}^{*+|\eta|-1}(X,k)}{
\widehat{\Ext}_{\widehat{A}}^{*-1}(X,k) \circ \theta} \to
\widehat{\Ext}_{\widehat{A}}^*(K_{\theta},k),$$
in which the index $*$ ranges over all the integers. Now let
$$\cdots \to P_2 \xrightarrow{d_2} P_1 \xrightarrow{d_1} P_0
\xrightarrow{d_0} P_{-1} \xrightarrow{d_{-1}} P_{-2} \to \cdots$$
be a complete resolution of $X$, and consider the group
$\widehat{\Ext}_{\widehat{A}}^{-n}(X,k)$ for any nonnegative integer $n$.
Since $\widehat{\Ext}_{\widehat{A}}^{-n}(X,k) =
\widehat{\Ext}_{\widehat{A}}^1( \Ker d_{-(n+2)},k)=
\Ext_{\widehat{A}}^1( \Ker d_{-(n+2)},k)$, the ring $\widehat{A} [
\chi ]$ acts on $\widehat{\Ext}_{\widehat{A}}^{-n}(X,k)$ from both
sides, and these actions coincide. Therefore
$\widehat{\Ext}_{\widehat{A}}^{*}(X,k) \circ \theta$ is an
$\widehat{A} [ \chi ]$-submodule of
$\widehat{\Ext}_{\widehat{A}}^{*+|\eta|}(X,k)$, and the above short
exact sequence is a sequence of $\widehat{A} [ \chi ]$-modules.

Now recall from the discussion prior to this theorem that
$\widehat{\Ext}_{\widehat{A}}^*(W,Z)_{\overline{\mathfrak{p}}}
\simeq \Ext_{\widehat{A}}^*(W,Z)_{\overline{\mathfrak{p}}}$ for any
$\widehat{A}$-modules $W$ and $Z$ with $W$ MCM. As
$\overline{\mathfrak{p}}$ does not contain the annihilator of
$\Ext_{\widehat{A}}^* (K_{\theta},k)$, we see by localizing the
above short exact sequence at $\overline{\mathfrak{p}}$ that
$\widehat{\Ext}_{\widehat{A}}^*(X,k)_{\overline{\mathfrak{p}}} =
\left [ \widehat{\Ext}_{\widehat{A}}^*(X,k) \circ \theta \right
]_{\overline{\mathfrak{p}}}$. Since $\theta = \phi_X( \overline{\eta}
) - \sum m_i \theta_i$ in the ring $\Ext_{\widehat{A}}^*(X,X)$, and
the ideal $\overline{\mathfrak{p}}$ contains both $\overline{\eta}$
and the $m_i$ (it contains the element $\overline{\eta}$ by
assumption, and contains the ideal $\widehat{\m}$ because it
corresponds to the ideal $\mathfrak{p} \subseteq H$ under the
isomorphism $\widehat{A} [ \chi ] \otimes_{\widehat{A}} k \simeq H$),
the $\widehat{A}( \chi )_{\overline{\mathfrak{p}}}$-module
$\left [ \widehat{\Ext}_{\widehat{A}}^*(X,k) \circ \theta \right
]_{\overline{\mathfrak{p}}}$ must be contained in $\overline{\mathfrak{p}}
\widehat{A}( \chi )_{\overline{\mathfrak{p}}} \cdot
\widehat{\Ext}_{\widehat{A}}^*(X,k)_{\overline{\mathfrak{p}}}$.
Consequently the inclusions
$$\widehat{\Ext}_{\widehat{A}}^*(X,k)_{\overline{\mathfrak{p}}}
\subseteq \overline{\mathfrak{p}}
\widehat{A}( \chi )_{\overline{\mathfrak{p}}} \cdot
\widehat{\Ext}_{\widehat{A}}^*(X,k)_{\overline{\mathfrak{p}}}
\subseteq
\widehat{\Ext}_{\widehat{A}}^*(X,k)_{\overline{\mathfrak{p}}}$$
hold, and so $\widehat{\Ext}_{\widehat{A}}^*(X,k)_{\overline{\mathfrak{p}}}
= \overline{\mathfrak{p}}
\widehat{A}( \chi )_{\overline{\mathfrak{p}}} \cdot
\widehat{\Ext}_{\widehat{A}}^*(X,k)_{\overline{\mathfrak{p}}}$.
But
$\widehat{\Ext}_{\widehat{A}}^*(X,k)_{\overline{\mathfrak{p}}}$,
being isomorphic to
$\Ext_{\widehat{A}}^*(X,k)_{\overline{\mathfrak{p}}}$, is finitely
generated over $\widehat{A} [ \chi
]_{\overline{\mathfrak{p}}}$, hence Nakayama's Lemma implies
$\Ext_{\widehat{A}}^*(X,k)_{\overline{\mathfrak{p}}}=0$. This
contradicts the assumption that $\overline{\mathfrak{p}}$ contains
the annihilator of $\Ext_{\widehat{A}}^* (X,k)$, and therefore
$\overline{\mathfrak{p}}$ must contain the annihilator of
$\Ext_{\widehat{A}}^* (K_{\theta},k)$. But then $\Ann_H
E(K_{\theta},k) \subseteq \mathfrak{p}$, giving the inclusion
$$\sqrt{\Ann_H E(K_{\theta},k)} \subseteq \sqrt{\left ( \eta,
\Ann_H E(X,k) \right ) }$$ of ideals in $H$, and consequently we
get $\V(X) \cap \V_H (\eta) \subseteq \V(K_{\theta})$.
\end{proof}

Suppose now that we start with an $A$-module $M$, and consider its
completion $\widehat{M}$. Let $\eta \in H^+$ be a homogeneous element,
let $\overline{\eta} \in \widehat{A} [
\chi ]$ be a corresponding element,
and consider the element $\phi_{\widehat{M}} ( \overline{\eta} ) \otimes 1$ in
$\Ext_{\widehat{A}}^* ( \widehat{M}, \widehat{M} )
\otimes_{\widehat{A}} k$. Since
$\Ext_{\widehat{A}}^* ( \widehat{M}, \widehat{M} )$ is isomorphic to
$\Ext_A^*(M,M) \otimes_A \widehat{A}$, there is an isomorphism
$$\Ext_A^*(M,M) \otimes_A k \xrightarrow{\sim} \Ext_{\widehat{A}}^* (
\widehat{M}, \widehat{M} ) \otimes_{\widehat{A}} k$$
under which the image of an element $\theta \otimes 1 \in
\Ext_A^*(M,M) \otimes_A k$ is $\widehat{\theta} \otimes 1$. Hence
there exists a homogeneous element $\theta^M \in \Ext_A^{|\eta|}(M,M)$
such that $\widehat{\theta^M} \otimes 1$ equals the element
$\phi_{\widehat{M}}( \overline{\eta} ) \otimes 1$ in
$\Ext_{\widehat{A}}^* ( \widehat{M},
\widehat{M} ) \otimes_{\widehat{A}} k$. Now if $\theta^M$ is
represented by the exact sequence
$$0 \to M \to K_{\theta^M} \to \Omega_A^{|\eta|-1}(M) \to 0,$$
then its completion $\widehat{\theta^M}$ is represented by the sequence
$$0 \to \widehat{M} \to \widehat{K_{\theta^M}} \to
\Omega_{\widehat{A}}^{|\eta|-1}(\widehat{M}) \to 0,$$
whose middle term \emph{is the completion of an $A$-module}. By
Theorem \ref{cuttingvariety} the inclusion
$$\V(K_{\theta^M}) \subseteq \V(M) \cap \V_H (\eta)$$
holds, with equality holding whenever $M$ is MCM, and consequently we
obtain the following corollary, showing that every homogeneous
algebraic set in $\tilde{k}^c$ is the variety of an MCM $A$-module.

\begin{corollary}\label{realizing}
Every closed homogeneous variety in $\tilde{k}^c$ is the variety
of some MCM $A$-module.
\end{corollary}

\begin{proof}
Let $\eta_1, \dots, \eta_t$ be homogeneous elements in $H^+$, and
denote by $M$ the MCM module $\Omega_A^{\dim A}(k)$. Then $\V(M) =
\V(k) = \tilde{k}^c$, and from the theorem and the above discussion we
see that there exists a homogeneous element $\theta_1 \in
\Ext_A^{|\eta_1|}(M,M)$ with the property that $\V(K_{\theta_1}) =
\V(M) \cap \V_H(\eta_1) = \V_H(\eta_1)$. Repeating the process
with $\eta_2, \dots, \eta_t$ we end up with an MCM $A$-module $K$ such
that
$$\V(K) = \V_H(\eta_1) \cap \cdots \cap \V_H(\eta_t) =
\V_H(\eta_1, \dots, \eta_t).$$
\end{proof}

\begin{remarks}
(i) In an unpublished preprint (as of December 2006), L.\ Avramov and
D.\ Jorgensen obtain a different proof of Corollary \ref{realizing},
based on a result concerning the realization of certain graded modules
as cohomology modules. L.\ Avramov reported it at a meeting at MSRI,
Berkeley, in December 2002.

(ii) In \cite{Erdmann} a realization theorem is proved for finite
dimensional algebras, and this result applies to complete
intersections containing a field (see also \cite[Section 7]{Snashall}).
\end{remarks}

\section{Decomposition}

Before proving the next result, recall that an
\emph{MCM-approximation} of an $A$-module $X$ is an exact sequence
$$0 \to Y_X \to C_X \xrightarrow{f} X \to 0$$
where $C_X$ is MCM and $Y_X$ has finite injective dimension. The
approximation is \emph{minimal} if the map $f$ is right minimal,
that is, if every map $C_X \xrightarrow{g} C_X$ satisfying $f=fg$ is
an isomorphism. This notion was introduced in \cite{Auslander2},
where it was shown that every finitely generated module over a
commutative Noetherian ring admitting a dualizing module has an
MCM-approximation. Moreover, it follows from the remark following
\cite[Theorem 18]{Martsinkovsky} that every finitely generated
module over a commutative local Gorenstein ring has a minimal
MCM-approximation, which is unique up to isomorphism. In particular
this applies to our setting, where $A$ is a local complete
intersection. Furthermore, since $A \to \widehat{A}$ is a faithfully
flat local homomorphism, an $A$-module $Z$ has finite projective
dimension if and only if the $\widehat{A}$-module $\widehat{Z}$ has
finite projective dimension, and it follows from \cite[Theorem
23.3]{Matsumura} that $Z$ is MCM if and only if $\widehat{Z}$ is
MCM. Therefore, by \cite[Proposition 19]{Martsinkovsky} and the fact
that over a Gorenstein ring the modules having finite injective
dimension are precisely those having finite projective dimension, we
see that
$$0 \to Y_X \to C_X \xrightarrow{f} X \to 0$$
is a minimal MCM-approximation if and only if
$$0 \to \widehat{Y}_X \to \widehat{C}_X \xrightarrow{\widehat{f}}
\widehat{X} \to 0$$ is a minimal MCM-approximation.

We are now ready to prove the second main result. It is the
commutative complete intersection version of J.\ Carlson's famous
theorem (see \cite{Carlson}) from modular representation theory;
if the variety $V$ of a $kG$-module $L$ (where $k$ is an
algebraically closed field and $G$ is a finite group) decomposes
as $V = V_1 \cup V_2$, where $V_1$ and $V_2$ are closed varieties
having trivial intersection, then $L$ decomposes as $L = L_1
\oplus L_2$ where the variety of $L_i$ is $V_i$. Our proof follows
closely that of J.\ Carlson, but with some adjustments.

\begin{theorem}\label{thm3.1}
If for an $A$-module $M$ we have $\V (M) = V_1 \cup V_2$ where $V_1$
and $V_2$ are closed homogeneous varieties having trivial
intersection, then the completion $\widehat{C}_M$ of the minimal
MCM-approximation of $M$ decomposes as $\widehat{C}_M = C_1 \oplus
C_2$ with $\V (C_i) = V_i$.
\end{theorem}

\begin{proof}
Let
$$0 \to Y \to C \to M \to 0$$
be the minimal MCM-approximation of $M$. Since $Y$ has finite
injective dimension (or equivalently, finite projective dimension),
it follows from \cite[Theorem 5.6]{Avramov2} that $\V (Y)$ is
trivial and that we therefore have $\V (M) = \V (C)$. Moreover, by
definition the equality $\V (X) = \V ( \widehat{X} )$ holds for
every $A$-module $X$, and therefore we may suppose that $A$ is
complete.

We argue by induction on the integer $\dim V_1 + \dim V_2$. If one
of $V_1$ and $V_2$, say $V_2$, is zero dimensional, then $V_2$ is
trivial, and the decomposition $C = C' \oplus P$, with $P$ being
the maximal projective summand of $C$, satisfies the conclusion of
the theorem. Suppose therefore that $\dim V_i$ is nonzero for $i
=1,2$.

Let $\az_1$ and $\az_2$ be homogeneous ideals of $H = k [ \chi ]$
defining the varieties $V_1$ and $V_2$, i.e.\
$V_i$ is the algebraic set $\V_H ( \az_i )$ in $\tilde{k}^c$
defined by $\az_i$ for $i=1,2$. We then have equalities
$$\{ 0 \} = V_1 \cap V_2 = \V_H ( \az_1 ) \cap \V_H ( \az_2 ) =
\V_H ( \az_1 + \az_2 ),$$ and so it follows from Hilbert's
Nullstellensatz that for each $1 \leq i \leq c$ we have $\chi_i
\in \sqrt{ \az_1 + \az_2}$. Therefore $\sqrt{ \az_1 + \az_2}$ is
the graded maximal ideal $H^+$ of $H$, i.e.\ $\sqrt{ \az_1 +
\az_2} = ( \chi )$.

Pick a homogeneous element $\theta \in H^+$ with the property that
$\dim H/(\az_2, \theta) < \dim H / \az_2$ (this is possible since
$\dim H/\az_2 = \dim V_2 >0$). By the above there is an integer $n
\geq 1$ such that $\theta^n$ belongs to $\az_1 + \az_2$, i.e.\
$\theta^n = \theta_1 + \eta$ where $\theta_1 \in \az_1$ and $\eta
\in \az_2$. Then $\dim H/( \az_2,\theta_1 ) < \dim H/\az_2$, which
translates to the language of varieties as $\dim \left ( \V_H (
\az_2 ) \cap \V_H ( \theta_1 ) \right ) = \dim \V_H ( \az_2 +
(\theta_1) ) < \dim \V_H ( \az_2 )$. Similarly we can find an
element $\theta_2 \in \az_2$ having the property that it ``cuts
down" the variety defined by $\az_1$. Hence the two homogeneous
elements $\theta_1$ and $\theta_2$ satisfy
\begin{eqnarray*}
\theta_1 \in \az_1, \hspace{.3cm} \dim \left ( V_2 \cap \V_H
(\theta_1)
\right ) < \dim V_2, \\
\theta_2 \in \az_2, \hspace{.3cm} \dim \left ( V_1 \cap \V_H
(\theta_2) \right ) < \dim V_1.
\end{eqnarray*}

\sloppy Now since $\V_H (\theta_1 \theta_2)= \V_H (\theta_1) \cup
\V_H (\theta_2) \supseteq V_1 \cup V_2 = \V (C)$, it follows once
more from Hilbert's Nullstellensatz that $\theta_1 \theta_2 \in
\sqrt{ \Ann_H E(C,C)}$, where $E(C,C) = \Ext_A^* (C,C) \otimes_A
k$. Replacing $\theta_1$ and $\theta_2$ by suitable powers, we may
assume that $\theta_1 \theta_2 \in \Ann_H E(C,C)$. Viewed as
elements in $A [ \chi ] \otimes_A k$ we have
$\theta_i = \overline{\theta}_i \otimes 1$, where
$\overline{\theta}_1$ and $\overline{\theta}_2$ are homogeneous
elements of positive degrees in $A [ \chi ]$ with
the property that $\overline{\theta}_1 \overline{\theta}_2 \in
\Ann_{A [ \chi ]} \Ext_A^* (C,C)$. To see the
latter, note that $0= \theta_1 \theta_2 \left ( \Ext_A^i (C,C)
\otimes_A k \right ) = \overline{\theta}_1 \overline{\theta}_2
\Ext_A^i (C,C) \otimes_A k$ for every $i \geq 0$, and since
$\overline{\theta}_1 \overline{\theta}_2 \Ext_A^i (C,C)$ is a
finitely generated $A$-module ($\overline{\theta}_1
\overline{\theta}_2$ commutes with elements in $A$), the claim
follows.

Now consider the images $\theta_1^C$ and  $\theta_2^C$ of
$\overline{\theta}_1$ and $\overline{\theta}_2$ in $\Ext_A^*
(C,C)$. Since $\theta_1^C \theta_2^C =0$, the bottom exact
sequence in the exact commutative diagram
$$\xymatrix{
0 \ar[r] & \Omega_A^{|\theta_1^C|+|\theta_2^C|}(C)
\ar[d]^{f_{\theta_1^C \theta_2^C}} \ar[r] &
Q_{|\theta_1^C|+|\theta_2^C|-1} \ar[d] \ar[r] &
\Omega_A^{|\theta_1^C|+|\theta_2^C|-1}(C) \ar@{=}[d] \ar[r] & 0 \\
0 \ar[r] & C \ar[r] & K_{\theta_1^C \theta_2^C} \ar[r] &
\Omega_A^{|\theta_1^C|+|\theta_2^C|-1}(C) \ar[r] & 0 }$$ splits,
where $Q_n$ denotes the $n$'th module in the minimal free resolution
of $C$. Therefore $K_{\theta_1^C \theta_2^C}$ is isomorphic to $C
\oplus \Omega_A^{|\theta_1^C|+|\theta_2^C|-1}(C)$, and from Lemma
\ref{link} we see that there exists an exact sequence
\begin{equation*}\label{ES3}
0 \to \Omega_A^{|\theta_1^C|}(K_{\theta_2^C}) \to C \oplus
\Omega_A^{|\theta_1^C|+|\theta_2^C|-1}(C) \oplus F \to
K_{\theta_1^C} \to 0 \tag{$\dagger$}
\end{equation*}
for some free module $F$. From Theorem \ref{cuttingvariety} we have
$\V (K_{\theta_i^C}) = \V (C) \cap \V_H ( \theta_i )$, hence the
equality $\V (C) = V_1 \cup V_2$ and the inclusion $V_i \subseteq
\V_H (\theta_i)$ give the equalities
\begin{eqnarray*}
\V (K_{\theta_1^C}) &=&  V_1 \cup \left ( V_2 \cap \V_H (
\theta_1) \right ),
\\
\V (K_{\theta_2^C}) &=& V_2 \cup \left ( V_1 \cap \V_H (\theta_2)
\right ).
\end{eqnarray*}
By induction there exist $A$-modules $X_1, X_2, Y_1$ and $Y_2$
such that $K_{\theta_1^C} = X_1 \oplus X_2$ and
$\Omega_A^{|\theta_1^C|}(K_{\theta_2^C}) = Y_1 \oplus Y_2$, and
such that
\begin{eqnarray*}
\V (X_1) &=& V_1, \\
\V (X_2) &=& V_2 \cap \V_H (\theta_1), \\
\V (Y_1) &=& V_1 \cap \V_H (\theta_2), \\
\V (Y_2) &=& V_2.
\end{eqnarray*}
Now since $\V (X_1) \cap \V (Y_2)$ and $\V (X_2) \cap \V (Y_1)$
are contained in $V_1 \cap V_2$, which is trivial, we see from
\cite[Theorem 5.6]{Avramov2} that $\Ext_A^i (X_1,Y_2)$ and
$\Ext_A^i (X_2,Y_1)$ vanish for $i \gg 0$. But $K_{\theta_1^C}$ is
MCM, implying $X_1$ and $X_2$ are both MCM, and so it follows from
\cite[Theorem 4.2]{Araya} that $\Ext_A^i (X_1,Y_2)$ and $\Ext_A^i
(X_2,Y_1)$ vanish for $i \geq 1$. Therefore
$$\Ext_A^1(K_{\theta_1^C},\Omega_A^{|\theta_1^C|}(K_{\theta_2^C}))
=\Ext_A^1(X_1,Y_1) \oplus \Ext_A^1(X_2,Y_2),$$ and this implies
that the exact sequence (\ref{ES3}) is equivalent to the direct
sum of two sequences of the form
$$0 \to Y_i \to Z_i \to X_i \to 0$$
for $i=1,2$, where $Z_i$ is an $A$-module. Then $C \oplus
\Omega_A^{|\theta_1^C|+|\theta_2^C|-1}(C) \oplus F$ must be
isomorphic to $Z_1 \oplus Z_2$, and since $\V (Z_i) \subseteq \V
(X_i) \cup \V (Y_i) \subseteq V_i$ and the Krull-Schmidt property
holds for the category of (finitely generated) modules over a
complete local ring, there must exist $A$-modules $C_1$ and $C_2$
such that $C = C_1 \oplus C_2$ and $\V (C_i) = \V (Z_i)$. Since
$$V = \V (C_1) \cup \V (C_2) \subseteq V_1 \cup V_2 = V$$
we must have $\V (C_i) = V_i$, and the proof is complete.
\end{proof}

\begin{corollary}\label{cor3.2}
The projective variety of a complete indecomposable MCM $A$-module
is connected.
\end{corollary}

\section*{Acknowledgements}

I would like to express my gratitude to Lucho Avramov for numerous
comments and improvements. Also, I would like to thank Dave Jorgensen
and my supervisor {\O}yvind
Solberg for valuable suggestions and comments on this paper.


\begin{thebibliography}{EHSST}
\bibitem[AuB]{Auslander2}M.\ Auslander, R.-O.\ Buchweitz,
\emph{The homological theory of maximal Cohen-Macaulay
approximations}, M\'{e}m.\ Soc.\ Math.\ France 38 (1989), 5-37.
\bibitem[Avr]{Avramov1}L.\ Avramov, \emph{Modules of finite
virtual projective dimension}, Invent.\ Math.\ 96 (1989), 71-101.
\bibitem[AvB]{Avramov2}L.\ Avramov, R.-O.\ Buchweitz,
\emph{Support varieties and cohomology over complete
intersection}, Invent.\ Math.\ 142 (2000), 285-318.
\bibitem[AGP]{Avramov3}L.\ Avramov, V.\ Gasharov, I.\ Peeva,
\emph{Complete intersection dimension}, Publ.\ Math.\ I.H.E.S.\ 86
(1997), 67-114.
\bibitem[ArY]{Araya}T.\ Araya, Y.\ Yoshino, \emph{Remarks on a
depth formula, a grade inequality and a conjecture of Auslander},
Comm.\ Algebra 26 (1998), 3793-3806.
\bibitem[Ber]{Bergh}P.A.\ Bergh, \emph{Modules with reducible
complexity}, J.\ Algebra 310 (2007), 132-147.
\bibitem[Car]{Carlson}J.\ Carlson, \emph{The variety of an
indecomposable module is connected}, Invent.\ Math.\ 77 (1984),
291-299.
\bibitem[EHSST]{Erdmann}K.\ Erdmann, M.\ Holloway, N.\ Snashall,
\O.\ Solberg, R.\ Taillefer, \emph{Support varieties for
selfinjective algebras}, $K$-theory 33 (2004), 67-87.
\bibitem[Mar]{Martsinkovsky}A.\ Martsinkovsky,
\emph{Cohen-Macaulay modules and approximations}, in \emph{Trends
in Mathematics: Infinite Length Modules}, H.\ Krause and C.\
Ringel (edts), Birk\"{a}user Verlag (2000), 167-192.
\bibitem[Mat]{Matsumura}H.\ Matsumura, \emph{Commutative ring
theory}, Cambridge University Press, 2000.
\bibitem[SnS]{Snashall}N.\ Snashall, \O.\ Solberg, \emph{Support
varieties and Hochschild cohomology rings}, Proc.\ London Math.\
Soc.\ 88 (2004), 705-732.
\end{thebibliography}
\end{document}